\documentclass{article}

\usepackage{graphicx}
\usepackage{xcolor}
\usepackage{indentfirst}
\usepackage{amsmath,amsfonts,amsthm,amssymb,amsrefs}
\usepackage{mathrsfs,stmaryrd}
\usepackage[all]{xy}
\usepackage{hyperref}

\DeclareMathAlphabet{\mathsfsl}{OT1}{cmss}{m}{sl}

\newtheorem{thm}{Theorem}[section]
\newtheorem{lem}[thm]{Lemma}

\newtheorem{prop}[thm]{Proposition}

\theoremstyle{definition}
\newtheorem{defn}[thm]{Definition}
\newtheorem{notn}[thm]{Notation}
\newtheorem{rem}[thm]{Remark}
\newtheorem{exam}[thm]{Example}

\begin{document}

\title{Generation of sutured manifolds}

\date{}
\author{{\Large Yi NI}\\{\normalsize Department of Mathematics, Caltech, MC 253-37}\\
{\normalsize 1200 E California Blvd, Pasadena, CA
91125}\\{\small\it Email\/:\quad\rm yini@caltech.edu}}

\maketitle

\begin{abstract}
Given a compact, oriented, connected surface $F$, we show that the set of connected sutured manifolds $(M,\gamma)$ with $R_{\pm}(\gamma)\cong F$ is generated by the product sutured manifold $(F,\partial F)\times[0,1]$ through surgery triads. This result has applications in Floer theories of $3$--manifolds. The special case when $F=D^2$ or $S^2$ has been a folklore theorem, which has already been used by experts before. 
\end{abstract}

\section{Introduction}

Surgery triad is an important concept in various Floer theories of $3$--manifolds. A surgery triad is often associated with exact triangles of Floer homologies, which have been extremely useful in many applications to low-dimensional topology. 

One strategy of proving certain isomorphism theorems in Floer theories of $3$--manifolds is to use surgery exact triangles and functoriality to reduce the proof of the general case to checking the isomorphism for manifolds in a generating set. See \cite{Lin} for an exposition.\footnote{The author first learned this strategy from Peter Ozsv\'ath, who had learned it from Tom Mrowka. The first application of this strategy was contained in \cite{BMO}.} In order to apply this strategy, we often need certain finite generation results of $3$--manifolds. To explain in detail, we need a few definitions.

\begin{defn}
A {\it surgery triad} is a triple $(Y,Y_0,Y_1)$ of $3$--manifolds, such that there is a framed knot $K\subset Y$ such that $Y_i$ is the $i$--surgery on $K$, $i=0,1$. 
\end{defn}

\begin{defn}
A set $\mathcal S$ of $3$--manifolds is {\it closed under surgery triads}, if whenever two manifolds in a surgery triad belong to $\mathcal S$ the third manifold in the triad must also belong to $\mathcal S$.
\end{defn}

\begin{defn}
A set $\mathcal S$ of $3$--manifolds is {\it generated by a subset $\mathcal G$} through surgery triads, if $\mathcal S$ is the minimal set containing $\mathcal G$ which is closed under surgery triads. In this case, we also say the elements in $\mathcal S$ are {\it generated by $\mathcal G$} through surgery triads.
\end{defn}

In order to apply the aforementioned strategy to study closed, oriented, connected $3$--manifolds, we need the following
 folklore theorem.

\begin{thm}\label{thm:ClosedGeneration}
The set of closed, oriented, connected $3$--manifolds is generated by $\{S^3\}$ through surgery triads.
\end{thm}

Baldwin and Bloom \cite{BB} generalized the above theorem to get a finite generation result for bordered $3$--manifolds with one boundary component. See \cite{Culler} for further work on this topic.

The goal of this note is to prove a generation result for sutured manifolds. We first recall the definition of sutured manifolds.

\begin{defn}
A {\it sutured manifold} $(M,\gamma)$ is a compact oriented
3--manifold $M$ together with a (possibly empty) set $\gamma\subset \partial M$ of
pairwise disjoint annuli. 

Every component of $R(\gamma)=\partial M-\mathrm{int}(\gamma)$ is
oriented. Define $R_+(\gamma)$ (or $R_-(\gamma)$) to be the union
of those components of $R(\gamma)$ whose normal vectors point out
of (or into) $M$. The orientations on $R(\gamma)$ must be coherent
with respect to the core of $\gamma$, hence every component of $\gamma$
lies between a component of $R_+(\gamma)$ and a component of
$R_-(\gamma)$.
\end{defn}

\begin{rem}
In Gabai's original definition of sutured manifolds \cite{G1}, $\gamma$ may also have torus components. We will not consider this case in the current paper.
\end{rem}

\begin{exam}
Let $F$ be a compact oriented surface, 
\[P=F\times [0,1], \delta=(\partial F)\times [0,1],
R_-(\delta)=F\times\{0\}, R_+(\delta)=F\times\{ 1\},\] then $(P,\delta)$ is a sutured manifold. In this case we say that $(P,\delta)$ is a {\it product sutured manifold}.
\end{exam}

Let $F$ be a compact oriented connected surface with $p$ boundary components, where $p\ge0$. Let $\mathcal X=\mathcal X(F)$ be the set of connected sutured manifolds $(M,\gamma)$ with $R_{\pm}(\gamma)\cong F$. Let $(P,\delta)=(F,\partial F)\times[0,1]$ be the product sutured manifold in $\mathcal X$.

The main theorem in our paper is as follows.

\begin{thm}\label{thm:SuturedGeneration}
The set $\mathcal X$ is generated by $\{(P,\delta)\}$.
\end{thm}

Theorem~\ref{thm:ClosedGeneration} follows from the special case of Theorem~\ref{thm:SuturedGeneration}
when $p=0$ or $1$ and $g=0$.
As far as we know, this is the first publicly available proof of Theorem~\ref{thm:ClosedGeneration}.

\begin{rem}
When $p>0$, the above theorem should presumably follow from the work of Baldwin and Bloom \cite{BB} on the finite generation of bordered manifolds, if one carefully examines the generators of bordered manifolds. 
\end{rem}


\begin{notn}
Suppose that $N$ is a submanifold of $M$. Let $M\bbslash N$ be the complement of an open tubular neighborhood of $N$ in $M$.
\end{notn}

\vspace{5pt}\noindent{\bf Acknowledgements.}\quad  The author was
partially supported by NSF grant number DMS-1811900. We are grateful to Fan Ye for the collaboration which inspired this work. We also thank John Baldwin for helpful comments on an earlier draft.


\section{Preliminaries}

In this section, we will collect some basic concepts and results about generalized Heegaard splittings, adapted to the category of sutured manifolds.

\begin{defn}
Let $(M,\gamma)\in \mathcal X$. A {\it generalized sutured Heegaard splitting (GSHS)} of $(M,\gamma)$ is a sequence of mutually disjoint surfaces 
\begin{equation}\label{eq:GHS}
\Xi=(\Sigma_1,F_1,\Sigma_2,F_2,\dots,\Sigma_{n-1},F_{n-1},\Sigma_n)
\end{equation}
properly embedded in $M$, with the following property: 
$(M,\gamma)$ is the union of $2n$ sutured manifolds 
\[(U_i,\mu_i),(V_i,\nu_i),\quad i=1,\dots,n,\] such that the interiors of these manifolds are mutually disjoint, and 
\[
R_-(\mu_i)=F_{i-1},\quad R_+(\mu_i)=\Sigma_i=R_-(\nu_i),\quad R_+(\nu_i)=F_i,
\]
where $F_0=R_-(\gamma)$, $F_n=R_+(\gamma)$. Moreover, $U_i$ is obtained by adding $1$--handles to $F_{i-1}\times[0,1]$ with feet on $\mathrm{int}(F_{i-1}\times\{1\})$, and $V_i$ is obtained by adding $2$--handles to $\Sigma_i\times[0,1]$ with attaching curves on $\mathrm{int}(\Sigma_i\times\{1\})$.
The manifolds $U_i,V_i$ are called {\it sutured compression bodies}.
\end{defn}

When $n=1$, the above definition is essentially the ``Heegaard splitting for sutured manifolds'' defined by Goda \cite{Goda}.

\begin{defn}
A GSHS $\Xi$ as in (\ref{eq:GHS}) is {\it connected} if every $F_i$ is connected.
\end{defn}

Every manifold in $ \mathcal X$ has a connected GSHS. In fact, we can choose $n=1$, in which case the GSHS is always connected.

\begin{defn}
Given a compact, oriented surface $S$, define
\[
c(S)=\sum_{S_j \text{ is a component of }S}\max\{1-\chi(S_j),0\}.
\]
\end{defn}

\begin{defn}
Let $(M,\gamma)\in \mathcal X$, with a GSHS as in (\ref{eq:GHS}). 
For each $\Sigma_i$, {\it a pair of non-separating attaching curves (a pair of NSACs)} in $\Sigma_i$ is a pair $(\alpha,\beta)$, where $\alpha\subset\Sigma_i$ is a non-separating simple closed curve bounding a disk in $U_i$,
and $\beta\subset\Sigma_i$ is a non-separating simple closed curve bounding a disk in $V_i$. 
Define 
\[
\iota(\Sigma_i)=\min\left\{|\alpha\cap\beta|\big|
\:(\alpha,\beta) \text{ is a pair of NSACs in }\Sigma_i
\right\}.
\]
When there does not exist any pair of NSACs, define $\iota(\Sigma_i)=\infty$.
The {\it Heegaard complexity} $hc(\Sigma_i)$ of $\Sigma_i$ is defined to be 
\[
(c(\Sigma_i),\iota(\Sigma_i)).
\]
\end{defn}

\begin{defn}\label{defn:HC}
Given a GSHS $\Xi$ as in (\ref{eq:GHS}), suppose that $\tau$ is a permutation of $1,2,\dots,n$, such that
\[
hc(\Sigma_{\tau(1)})\ge hc(\Sigma_{\tau(2)})\ge \cdots\ge hc(\Sigma_{\tau(n)}),
\]
where the order is the lexicographic order.
Let 
\[
HC(\Xi)=\big(
hc(\Sigma_{\tau(1)}), hc(\Sigma_{\tau(2)}), \dots, hc(\Sigma_{\tau(n)})\big),
\]
be the {\it Heegaard complexity} of $\Xi$. 

For $(M,\gamma)\in \mathcal X$, define its {\it Heegaard complexity}
\[
HC(M,\gamma)=\min\{HC(\Xi)| \Xi \text{ is a connected GSHS of }(M,\gamma)\},
\]
where the order is the lexicographic order.
\end{defn}

\begin{exam}\label{exam:MinHC}
The product sutured manifold $(P,\delta)$ has the trivial Heegaard splitting $\Omega$ with neither $1$--handles nor $2$--handles. We have
\[
HC(\Omega)=((c(F),\infty)).
\]
Suppose that $(M,\gamma)\in \mathcal X$, and $\Xi$ as in (\ref{eq:GHS}) is a GSHS of $(M,\gamma)$. Then we always have $c(\Sigma_1)\ge c(F)$, and the equality holds if and only if $U_1$ is a product sutured manifold, in which case we have $\iota(\Sigma_1)=\infty$. As a result, we always have $HC(\Xi)\ge HC(\Omega)$.

If $HC(\Xi)=HC(\Omega)$, then $n=1$ and $c(\Sigma_1)=c(F)$, which means that both $U_1$ and $V_1$ are product sutured manifolds. Hence $\Xi=\Omega$.

It follows that 
\[
HC(P,\delta)=HC(\Omega)\le HC(M,\gamma),
\]
and the equality holds if and only if $(M,\gamma)$ is the product sutured manifold.
\end{exam}

\begin{rem}
Generalized Heegaard splittings for compact oriented $3$--manifolds are defined by Scharlemann and Thompson \cite{ST_ThinPosition}.
Our Heegaard complexity of a GSHS is a refinement of the {\it width} defined in \cite{ST_ThinPosition} in the sutured category. Note that the generalized Heegaard splittings in Scharlemann and Thompson's definition of the width need not to be connected.
\end{rem}

\begin{lem}\label{lem:FramedSurg}
Suppose that $(M,\gamma)$ has a GSHS $\Xi$ as in (\ref{eq:GHS}), and $K\subset \Sigma_i$ is a knot with the surface framing. Let $\Sigma_i^-$ be a parallel copy of $\Sigma_i$ below $\Sigma_i$, and let $\Sigma_i^+$ be a parallel copy of $\Sigma_i$ above $\Sigma_i$. Let $K^{\pm}\subset\Sigma_i^{\pm}$ be the copies of $K$ in $\Sigma_i^{\pm}$.
Let $M'$ be the sutured manifold obtained by $0$--surgery on $K$, and let $F'\subset M'$ be the surface obtained from $\Sigma_i\bbslash K$ by capping off the two boundary components parallel to $K$ with disks.
Then $M'$ has a GSHS
\begin{equation}\label{eq:Xi'}
\Xi'=(\Sigma_1,F_1,\dots,F_{i-1},\Sigma_i^-,F',\Sigma_i^+,F_i,\dots, \Sigma_{n-1},F_{n-1},\Sigma_n).
\end{equation}
The sutured compression body between $F_{i-1}$ and $\Sigma_i^-$ is isotopic to $U_i$ in $M\bbslash K$, and the sutured compression body between $\Sigma_i^-$ and $F'$ is obtained by adding a 
$2$--handle to  $\Sigma_i^-\times[0,1]$, with attaching curve $K^-$. Similarly, the sutured compression body between $F'$ and $\Sigma_i^+$ is obtained by adding a 
$1$--handle to  $F'\times[0,1]$, with belt curve $K^+$, and
the sutured compression body between $\Sigma_i^+$ and $F_i$ is isotopic to $V_i$ in $M\bbslash K$.

Moreover, if $\Xi$ is connected and $K$ is non-separating in $\Sigma_i$, then $\Xi'$ is also connected.
\end{lem}
\begin{proof}
The two new disks in $F'$ split the new solid torus in $M'$ into two $D^2\times[0,1]$'s, which may be viewed as $1$--handles attached to two different sides of $F'$. The belt curves of these two $1$--handles are $K^{\pm}$. The $1$--handle attached to the bottom side of $F'$ can be viewed as a $2$--handle attached to the top side of $\Sigma_i^-$. So we get a GSHS as in the statement of this lemma.

If $\Xi$ is connected and $K$ is non-separating in $\Sigma_i$, then $F'$ is also connected, so $\Xi'$ is connected.
\end{proof}

The following lemma is well-known.

\begin{lem}\label{lem:pm1Surg}
Suppose that $(M,\gamma)$ has a GSHS as in (\ref{eq:GHS}), and $K\subset \Sigma_i$ is a knot with the surface framing.
We push $K$ slightly into $U_i$ and do $(+1)$--surgery (resp., $(-1)$--surgery) on it, the resulting manifold is denoted by $M''$. Then $M''$ has a GSHS
\begin{equation}\label{eq:Xi''}
\Xi''=(\Sigma_1,F_1,\dots,F_{i-1},\Sigma_i'',F_i,\dots, \Sigma_{n-1},F_{n-1},\Sigma_n),
\end{equation}
where $\Sigma_i''$ is isotopic to $\Sigma_i$ in $M\bbslash K$. The only difference between $\Xi$ and $\Xi''$ is, the curves on $\Sigma_i''$ which bounds disks in $U_i$ are exactly the images of the corresponding curves in $\Sigma_i$ under the left-handed (resp. right-handed) Dehn twist along $K$.
\end{lem}


\section{Proof of the main theorem}

In this section, we will prove Theorem~\ref{thm:SuturedGeneration}. We will induct on $HC(M,\gamma)$, and the inductive step is the following proposition.

\begin{prop}\label{prop:GenerateM}
Let $(M,\gamma)\in \mathcal X$ be a non-product sutured manifold, then $(M,\gamma)$ is generated by 
\[
\mathcal X^{<M}=\{(N,\zeta)\in \mathcal X|\:HC(N,\zeta)<HC(M,\gamma)\}.
\]
\end{prop}

From now on, let $(M,\gamma)\in \mathcal X$ be a non-product sutured manifold, and let $\Xi$ as in (\ref{eq:GHS}) be a connected GSHS of $(M,\gamma)$ with $HC(\Xi)=HC(M,\gamma)$.
Let $i=\tau(1)$, where $\tau$ is as in Definition~\ref{defn:HC}, and let $(\alpha,\beta)$ be a pair of NSACs in $\Sigma_i$ with $|\alpha\cap \beta|=\iota(\Sigma_i)$ if $\iota(\Sigma_i)<\infty$. We orient $\alpha,\beta$ arbitrarily.

\begin{lem}\label{lem:No1infty}
With the above notations, we have $\iota(\Sigma_i)\ne1,\infty$.
\end{lem}
\begin{proof}
If $|\alpha\cap \beta|=1=\iota(\Sigma_i)$, we can cancel the $1$--handle corresponding to $\alpha$ with the $2$--handle corresponding to $\beta$, thus decrease $c(\Sigma_i)$ and hence $HC(\Xi)$. The new GSHS has the same collection of $F_i$'s, so it is still connected.
This contradicts the choice of $\Xi$.

If $\iota(\Sigma_i)=\infty$, since $\Xi$ is connected, either there are no $1$--handles in $U_i$, or there are no $2$--handles in $V_i$. That is,
one of $U_i,V_i$ must be a product sutured manifold. 

Without loss of generality, we may assume $U_i$ is a product sutured manifold. Then the $2$--handles in $V_i$ may be viewed as attached to $\Sigma_{i-1}$. We then get a new connected GSHS of $(M,\gamma)$:
\[
\Xi'=(\Sigma_1,F_1,\dots,F_{i-2},\Sigma'_{i-1},F_i,\Sigma_{i+1},F_{i+1},\dots,\Sigma_n),
\]
where $\Sigma_{i-1}'$ is just $\Sigma_{i-1}$, but now there are additional attaching curves bounding disks in $V_{i-1}'$, the sutured compression body between $\Sigma'_{i-1}$ and $F_i$. So we have
\[
c(\Sigma_{i-1}')=c(\Sigma_{i-1}),\quad \iota(\Sigma_{i-1}')\le\iota(\Sigma_{i-1}),
\]
which implies $hc(\Sigma_{i-1}')\le hc(\Sigma_{i-1})$. 

Moreover, in $HC(\Xi')$, there is no component corresponding to the component $hc(\Sigma_i)$ in $HC(\Xi)$. So we have
\[
HC(\Xi')<HC(\Xi),
\]
a contradiction to the choice of $\Xi$.
\end{proof}

\begin{lem}\label{lem:StandardTriad}
Suppose that there exists a non-separating simple closed curve $K\subset\Sigma_i$, and a curve $\alpha_1$ which is the image of $\alpha$ under a (left-handed or right-handed) Dehn twist along $K$, satisfying the following two conditions:
\begin{itemize}
\item either 
\begin{equation}\label{eq:KIntersectLess}
\max\{|K\cap \alpha|,|K\cap \beta|\}<|\alpha\cap\beta|,
\end{equation}
or
\begin{equation}\label{eq:KIntersect1}
|K\cap \alpha|=|K\cap \beta|=1,
\end{equation}
\item $\alpha_1$ can be isotoped so that either
\begin{equation}\label{eq:alpha1IntersectLess}
|\alpha_1\cap\beta|<|\alpha\cap\beta|
\end{equation}
or 
\begin{equation}\label{eq:alpha1Intersect1}
|\alpha_1\cap\beta|=1.
\end{equation}
\end{itemize}
Then the conclusion of Proposition~\ref{prop:GenerateM} holds true for $(M,\gamma)$.
\end{lem}
\begin{proof}
Let $M'$ be the $0$--surgery on $K$, and let $M''$ be the $(+1)$ or $(-1)$--surgery on $K$, with the sign chosen so that the corresponding Dehn twist (as in Lemma~\ref{lem:pm1Surg}) applying to $\alpha$ yields $\alpha_1$.

By Lemma~\ref{lem:FramedSurg}, $M'$ has a GSHS $\Xi'$ as in (\ref{eq:Xi'}). Since $K$ is non-separating in $\Sigma$, $F'$ is still connected, so $\Xi'$ is connected.
On $\Sigma_i^-$, we have a pair of NSACs $(\alpha,K)$. On $\Sigma_i^+$, we have a pair of NSACs $(K,\beta)$. If (\ref{eq:KIntersectLess}) holds, then 
\[
\max\{
\iota(\Sigma_i^-),\iota(\Sigma_i^+)\}<\iota(\Sigma_i).
\]
If (\ref{eq:KIntersect1}) holds, then we can cancel the only $2$--handle between $\Sigma_i^-$ and $F'$ with the $1$--handle corresponding to $\alpha$, and cancel the only $1$--handle between $F'$ and $\Sigma_i^+$ with the $2$--handle corresponding to $\beta$. This gives us a new connected GSHS
\[
(\Sigma_1,F_1,\dots,F_{i-1},\Sigma_i^*,F_i,\dots, \Sigma_{n-1},F_{n-1},\Sigma_n)
\]
with $c(\Sigma_i^*)=c(F')<c(\Sigma_i)$. 
In either case, we have $HC(M')<HC(M)$.

By Lemma~\ref{lem:pm1Surg}, $M''$ has a connected GSHS $\Xi''$ as in (\ref{eq:Xi''}), and a pair of NSACs $(\alpha_1,\beta)$ on $\Sigma_i''$.
If (\ref{eq:alpha1IntersectLess}) holds, $\iota(\Sigma_i'')<\iota(\Sigma_i)$. If  (\ref{eq:alpha1Intersect1}) holds, we can cancel the corresponding $1$--handle/$2$--handle pair to decrease $c(\Sigma_i'')$. In either case,
we have $HC(M'')<HC(M)$. 

Now our conclusion holds true since $(M,M',M'')$ is a surgery triad.
\end{proof}

\begin{lem}\label{lem:NoIntersection}
If $\alpha\cap \beta=\emptyset$, then the conclusion of Proposition~\ref{prop:GenerateM} holds true for $(M,\gamma)$.
\end{lem}
\begin{proof}
Since $\alpha$ is non-separating, $\Sigma_i\bbslash \alpha$ is connected. Since $\beta$ is non-separating in $\Sigma_i$, $\Sigma_i\bbslash (\alpha\cup\beta)$ has $1$ or $2$ components, and none of the components has only one boundary component. In either case, we can find a simple closed curve $K$ which intersects each of $\alpha$ and $\beta$ exactly once. 

Let $\alpha_1$ be the image of $\alpha$ under the left-handed Dehn twist along $K$, as shown in Figure~\ref{fig:DehnTwist0}. Then $|\alpha_1\cap\beta|=1$.

Now our conclusion holds true by Lemma~\ref{lem:StandardTriad}.
\end{proof}

\begin{figure}[ht]
\begin{picture}(345,135)
\put(15,0){\scalebox{0.58}{\includegraphics*
{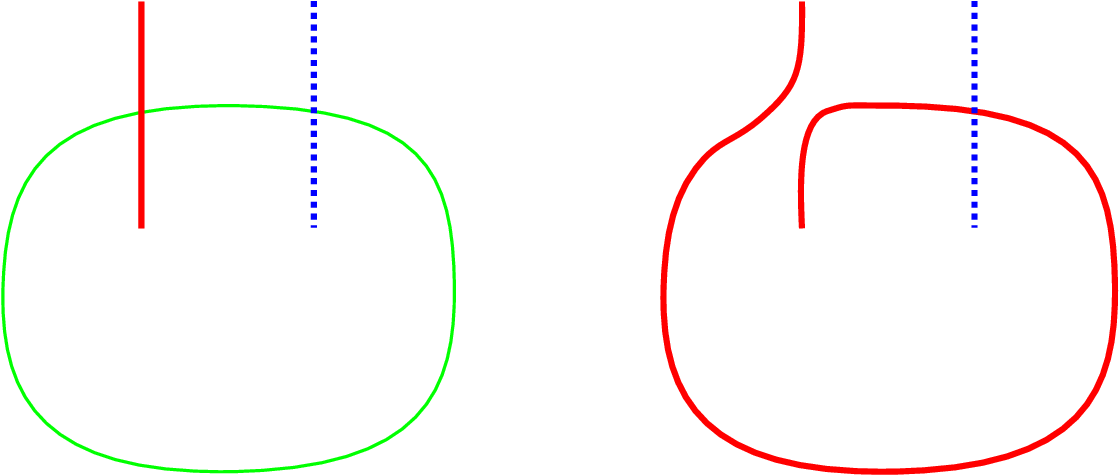}}}

\put(57,118){$\alpha$}

\put(105,118){$\beta$}

\put(18,46){$K$}

\put(241,118){$\alpha_1$}

\put(290,118){$\beta$}

\end{picture}
\caption{\label{fig:DehnTwist0}The local picture near $K$.}
\end{figure}
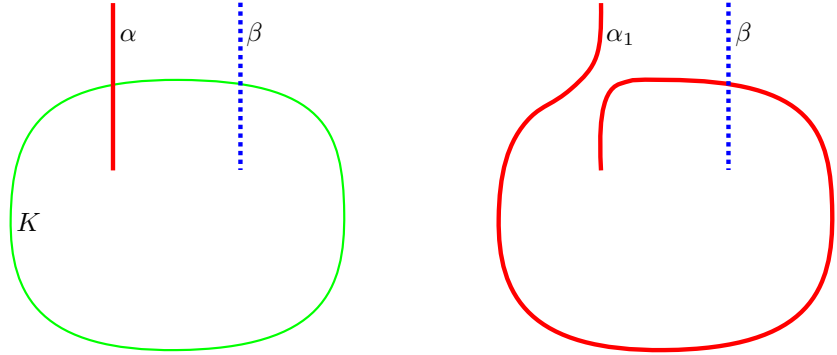

\begin{lem}\label{lem:TwoPoints}
If $\alpha\cap \beta$ has exactly two points, which have opposite signs, then the conclusion of Proposition~\ref{prop:GenerateM} holds true for $(M,\gamma)$.
\end{lem}
\begin{proof}
Let $x\in\alpha\cap \beta$. Consider the local picture near $x$ as in Figure~\ref{fig:Quadrants}. Since $|\alpha\cap \beta|=2$, every component of $\Sigma\setminus(\alpha\cup \beta)$ must contain at least one of the four (right) angles at $x$. 
We claim that there exists a pair of opposite angles at $x$, which are contained in the same component.

Assume that the claim is not true. Since $\alpha$ is non-separating, there exists a path in $\Sigma_i\setminus \alpha$ connecting two sides of $\alpha$. Note that whenever we cross $\beta$, we either move between the component containing Angle 1 and the component containing Angle 4, or move between the component containing Angle 2 and the component containing Angle 3.
It follows that one of Angle~1 and Angle~4 is in the same component as one of Angle~2 and Angle~3. By our assumption, these two angles cannot be opposite angles. Without loss of generality, we may assume Angle~1 and Angle~2 are in the same component. Similarly, since $\beta$ is non-separating, one of Angle~1 and Angle~2 is in the same component as one of Angle~3 and Angle~4. Since Angle~1 and Angle~2 are already in the same component, this forces a pair of opposite angles to be in the same component.

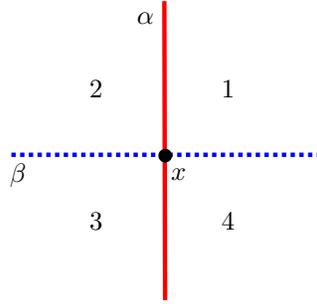
\begin{figure}[ht]
\begin{picture}(345,115)
\put(115,0){\scalebox{0.58}{\includegraphics*
{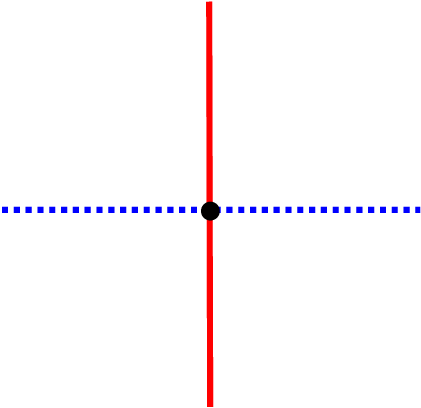}}}

\put(163,105){$\alpha$}

\put(176,46){$x$}

\put(115,45){$\beta$}

\put(195,77){$1$}

\put(145,77){$2$}

\put(145,27){$3$}

\put(195,27){$4$}

\end{picture}
\caption{\label{fig:Quadrants}The local picture near $x$. The four angles are labeled with $1,2,3,4$.}
\end{figure}

Using the claim, we can find a simple closed curve $K$ which intersects each of $\alpha$ and $\beta$ exactly once, and $K\cap \alpha=K\cap\beta=\{x\}$. As in the right part of Figure~\ref{fig:DehnTwist2}, the curve $\alpha_1$ is obtained from $\alpha$ by a Dehn twist along $K$, and $|\alpha_1\cap \beta|=1$. Now our conclusion follows from Lemma~\ref{lem:StandardTriad}.
\end{proof}

\begin{figure}[ht]
\begin{picture}(345,180)
\put(15,0){\scalebox{0.58}{\includegraphics*
{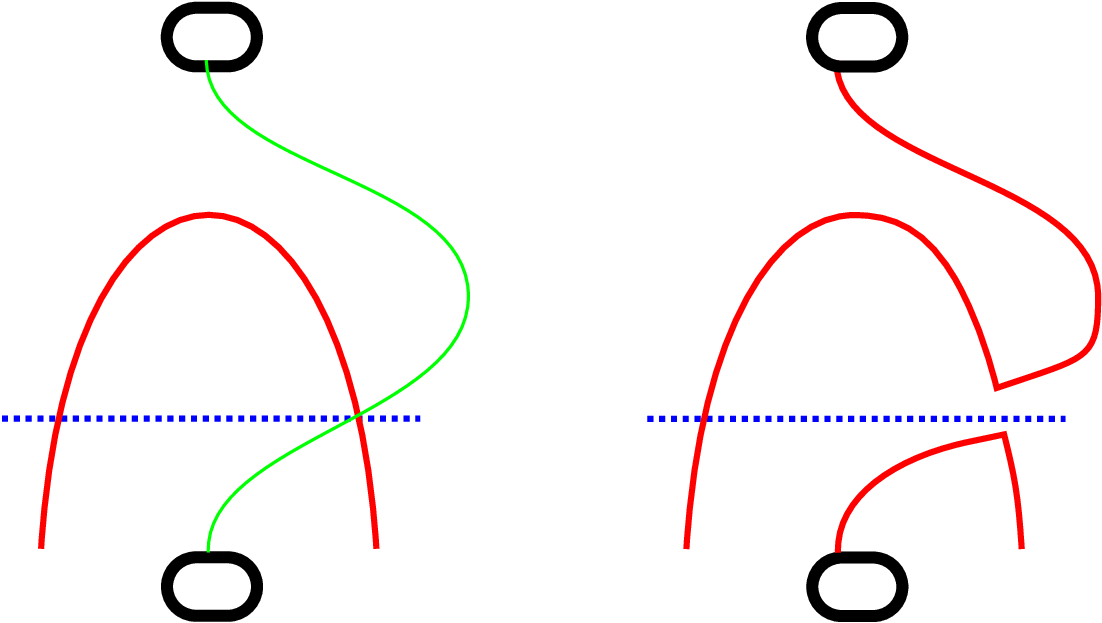}}}

\put(69,118){$\alpha$}

\put(107,130){$K$}

\put(12,46){$\beta$}

\put(253,118){$\alpha_1$}

\put(193,46){$\beta$}

\end{picture}
\caption{\label{fig:DehnTwist2}In each part of the picture, two ovals are glued together by a vertical reflection to form a tube in the surface.}
\end{figure}
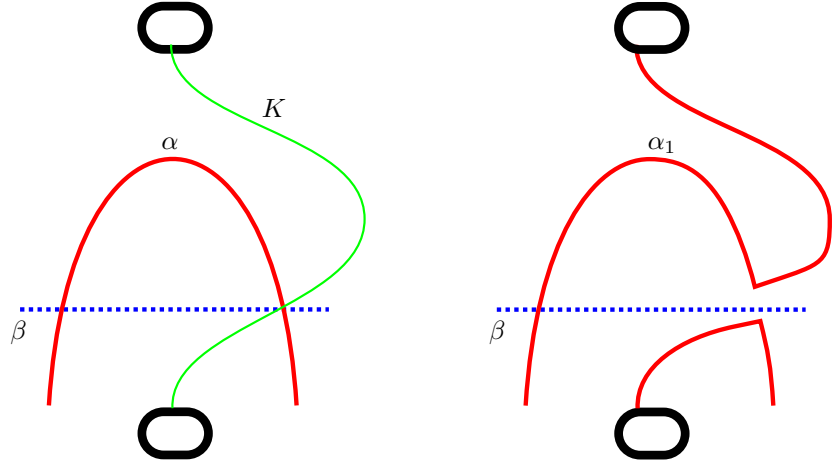

\begin{lem}\label{lem:SameSign}
If there exist two intersection points $x,y\in \alpha\cap \beta$ of the same sign, and an arc $a$ in $\alpha$ connecting $x$ to $y$, such that the interior of $a$ does not intersect $\beta$, 
then the conclusion of Proposition~\ref{prop:GenerateM} holds true for $(M,\gamma)$.
\end{lem}
\begin{proof}
Let $b$ be an arc in $\beta$ connecting $x$ to $y$, then $a\cap b=\{x,y\}$.

Let $K\subset \Sigma_i$ be a small perturbation of $a\cup b$, such that $K\cap(a\cup b)=\emptyset$. 

As in the left part of Figure~\ref{fig:DehnTwist1}, we clearly have 
\[|K\cap\alpha|<|\alpha\cap \beta|,\]
and 
\[|K\cap\beta|=1\le|\alpha\cap \beta|.\] 
Since $|K\cap\beta|=1$, $K$ is non-separating.

As in the right part of Figure~\ref{fig:DehnTwist1}, we can choose an appropriate Dehn twist along $K$, such that $\alpha_1$ can be isotoped so that (\ref{eq:alpha1IntersectLess}) holds.

Now our conclusion follows from Lemma~\ref{lem:StandardTriad}.
\end{proof}

\begin{figure}[ht]
\begin{picture}(345,172)
\put(15,0){\scalebox{0.58}{\includegraphics*
{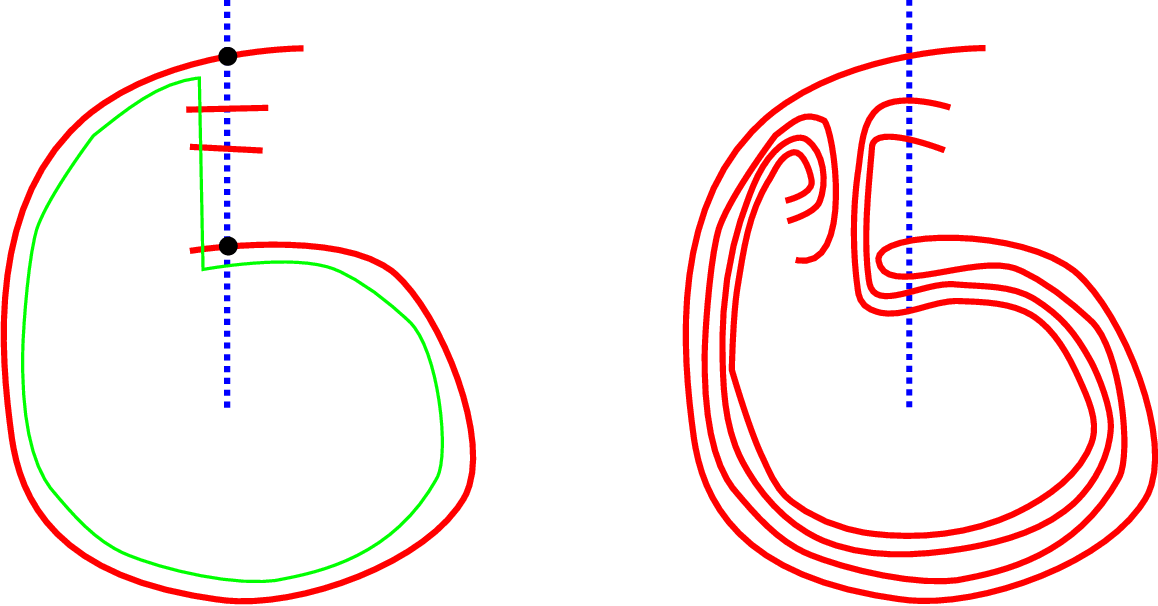}}}

\put(88,131){$\alpha$}

\put(16,118){$a$}

\put(81,110){$b$}

\put(24,50){$K$}

\put(202,118){$\alpha_1$}

\put(271,65){$\beta$}

\end{picture}
\caption{\label{fig:DehnTwist1}The two black dots are intersection points between $\alpha$ and $\beta$ of the same sign. 
The arc $b$ may also intersect $\alpha$ in its interior. In this case, $\alpha_1$ is obtained from $\alpha$ by a left-handed Dehn twist along $K$.}
\end{figure}
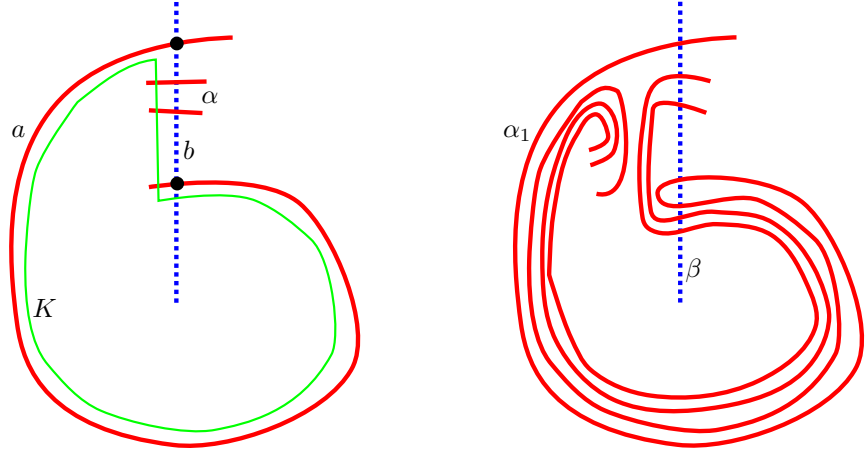

By Lemmas~\ref{lem:No1infty}, \ref{lem:NoIntersection}, \ref{lem:TwoPoints}, \ref{lem:SameSign}, in order to prove Proposition~\ref{prop:GenerateM}, the only remaining case is that the signs of the intersection points in $\alpha\cap \beta$ appear as positive and negative alternatively on $\alpha$, and there are more than $2$ intersection points.

Let $G=\Sigma_i\bbslash\beta$, with $\partial G=\beta_0\sqcup\beta_1$. Then $\alpha\cap G$ consists of arcs $a_1,\dots,a_{2m}$ for some $m>1$, in the cyclic order they appear on $\alpha$, with 
\[\partial a_j\subset \beta_k, \quad k=0,1, \quad j\equiv k\pmod2.\]
Moreover, $\alpha$ splits $\beta$ into arcs $b_1,\dots,b_{2m}$, in the cyclic order they appear on $\beta$. For each $j\in\{1,\dots,2m\}$, pick two points $z_j^0,z_j^1$, one by each side of $b_j$, such that $z_j^k$ is near  $\beta_k$ in $G$, $k=0,1$.


\begin{lem}\label{lem:Component}
There exists a component $A$ of $G\bbslash(\cup_{j}a_j)$, and two different indices $j_0,j_1\in \{1,\dots,2m\}$,
such that $z_{j_0}^0,z_{j_1}^1\in A$.
\end{lem}
\begin{proof}
Let
\[
G_0=G\bbslash(\bigcup_{j \text{ even}}a_j).
\]
Since $a_j\cap \beta_1=\emptyset$ when $j$ is even, there exists a distinguished component $A_0$ of $G_0$, such that $\beta_1\subset A_0$. Of course, $A_0\cap\beta_0\ne\emptyset$ since every $a_j$ intersects $\beta_0$ when $j$ is even. Let $A$ be a componet of 
\[
A_0\bbslash(\bigcup_{j \text{ odd}}a_j)
\]
which intersects $\beta_0$. Then $A$ also intersects $\beta_1$ since every $a_j$ intersects $\beta_1$ when $j$ is odd.
So $A\cap\beta_0\ne\emptyset$ and $A\cap\beta_1\ne\emptyset$. It follows that 
\[
A\cap\{z^k_1,\dots,z^k_{2m}\}\ne\emptyset, \text{ for each }k\in\{0,1\}.
\]

Assume that the conclusion of this lemma does not hold, then there exists $j\in\{1,\dots,2m\}$, such that 
\begin{equation}\label{eq:SinglePoint}
A\cap\{z^k_1,\dots,z^k_{2m}\}=\{z^k_j\}, \text{ for each }k\in\{0,1\}.
\end{equation}
Let $x\in\partial b_j$.  
Suppose that $x$ is the common endpoint of $a_{\ell}$ and $a_{\ell+1}$ for some $\ell$. Since the other endpoint of $a_{\ell}$ is also a corner of $A$, it follows from (\ref{eq:SinglePoint}) that $\partial a_{\ell}=\partial b_j$. The same argument shows that $\partial a_{\ell+1}=\partial b_j$.
This implies that $|\alpha\cap\beta|=2$, a contradiction to our assumption that $m>1$.
\end{proof}

Now we are ready to deal with the last case of Proposition~\ref{prop:GenerateM}.

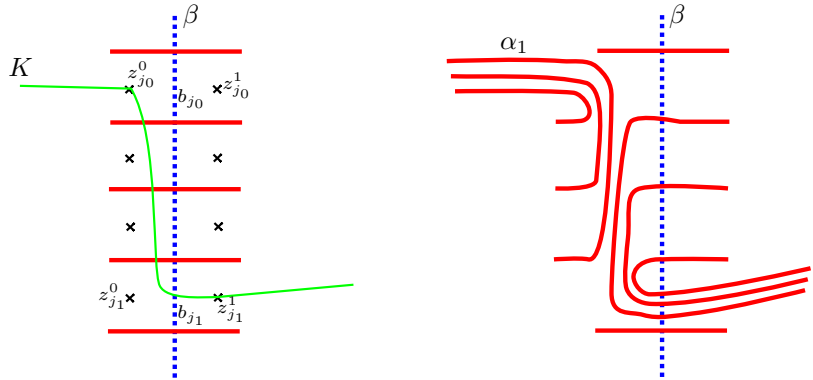
\begin{figure}[ht]
\begin{picture}(345,140)
\put(20,0){\scalebox{0.58}{\includegraphics*
{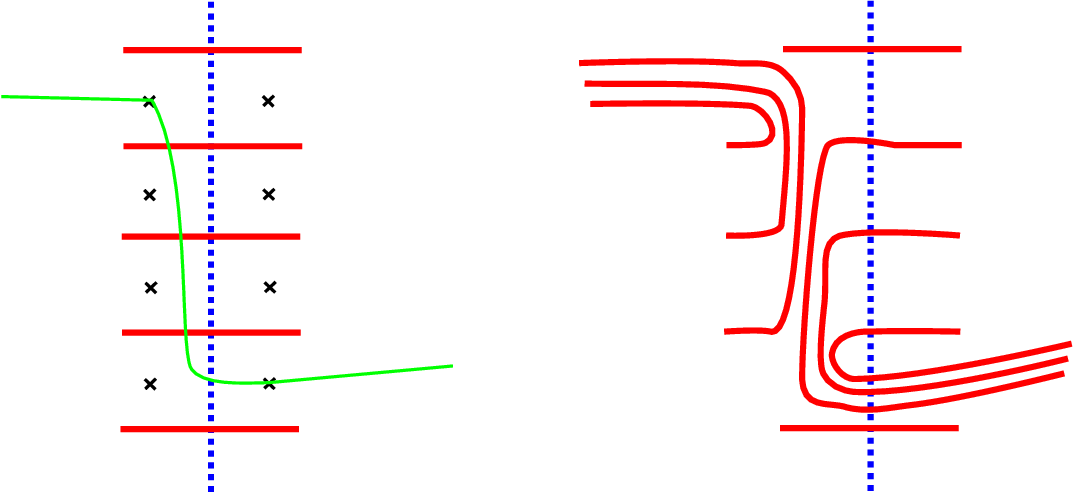}}}

\put(82,135){$\beta$}

\put(16,115){$K$}

\put(61,114){$\scriptstyle z^0_{j_0}$}

\put(97,110){$\scriptstyle z^1_{j_0}$}

\put(80,105){$\scriptstyle b_{j_0}$}

\put(50,30){$\scriptstyle z^0_{j_1}$}

\put(95,25){$\scriptstyle z^1_{j_1}$}

\put(80,23){$\scriptstyle b_{j_1}$}

\put(202,125){$\alpha_1$}

\put(266,135){$\beta$}

\end{picture}
\caption{\label{fig:AltSign}The local picture in a neighborhood of an arc in $\beta$.}
\end{figure}

\begin{lem}\label{lem:AltSign}
Suppose that the signs of the intersection points in $\alpha\cap\beta$ appear as positive and negative alternatively on $\alpha$, then the conclusion of Proposition~\ref{prop:GenerateM} holds true for $(M,\gamma)$.
\end{lem}
\begin{proof}
Let $z_{j_0}^0,z_{j_1}^1$ be as in Lemma~\ref{lem:Component}. Since they are in the same component of $G\bbslash(\cup_{j}a_j)$, we can connect them by an arc $e_1\subset\Sigma_i$ with $e_1\cap(\alpha\cup\beta)=\emptyset$. Since $j_0\ne j_1$, we can connect $z_{j_0}^0,z_{j_1}^1$ by an arc $e_2\subset\Sigma_i$ satisfying
\begin{equation}\label{eq:e2}
|e_2\cap\beta|=1,\quad0<d=|e_2\cap\alpha|<2m.
\end{equation}
Let $K=e_1\cup e_2$, then (\ref{eq:e2}) holds true if we replace $e_2$ with $K$. In particular, this implies that $K$ is non-separating.
See the left part of Figure~\ref{fig:AltSign}.

As in the right part of Figure~\ref{fig:AltSign}, we can apply a Dehn twist along $K$ to $\alpha$, to get a curve $\alpha_1$. After an isotopy, we get 
\[
|\alpha_1\cap\beta|=2m-d<2m.
\]

Now our conclusion follows from Lemma~\ref{lem:StandardTriad}.
\end{proof}

This finishes the proof of Proposition~\ref{prop:GenerateM}. Combining Proposition~\ref{prop:GenerateM} with Example~\ref{exam:MinHC}, we get
Theorem~\ref{thm:SuturedGeneration} by a straightforward induction.

%

\end{document}